\documentclass[a4paper,11pt,dvipdfmx]{amsart}
\usepackage[svgnames]{xcolor}
\usepackage{tikz}

\usepackage{comment}

\usepackage{etoolbox}
\ifdef{\crop}{%
\usepackage[includeheadfoot,twoside=False,paperwidth=448pt,paperheight=587pt,rmargin=15pt,lmargin=15pt,tmargin=15pt,bmargin=15pt]{geometry}%
}{%
\setlength{\topmargin}{22mm}
\addtolength{\topmargin}{-1in}
\setlength{\oddsidemargin}{27mm}
\addtolength{\oddsidemargin}{-1in}
\setlength{\evensidemargin}{27mm}
\addtolength{\evensidemargin}{-1in}
\setlength{\textwidth}{156mm}
\setlength{\textheight}{240mm}
}%

\usepackage[active]{srcltx}
\usepackage{amsmath,amsthm,amsxtra}
\usepackage{amssymb}

\usepackage[mathscr]{eucal}

\usepackage{bm}
\usepackage[all]{xy}

\usepackage{aliascnt}

\theoremstyle{plain}
\newtheorem{thm}{Theorem}[section]
\newtheorem*{thm*}{Theorem}

\newaliascnt{prop}{thm}
\newaliascnt{cor}{thm}
\newaliascnt{lem}{thm}
\newaliascnt{claim}{thm}
\newaliascnt{defn}{thm}
\newaliascnt{ques}{thm}
\newaliascnt{conj}{thm}
\newaliascnt{fact}{thm}
\newaliascnt{rem}{thm}
\newaliascnt{ex}{thm}
\newtheorem{prop}[prop]{Proposition}
\newtheorem{cor}[cor]{Corollary}
\newtheorem{lem}[lem]{Lemma}

\newtheorem*{prop*}{Proposition}
\newtheorem*{cor*}{Corollary}
\newtheorem*{lem*}{Lemma}
\newtheorem*{claim*}{Claim}
\theoremstyle{definition}
\newtheorem{defn}[defn]{Definition}
\newtheorem{ques}[ques]{Question}
\newtheorem{conj}[conj]{Conjecture}
\newtheorem*{defn*}{Definition}
\newtheorem*{ques*}{Question}
\newtheorem*{conj*}{Conjecture}

\newtheorem*{prob*}{Problem}

\newtheorem{rem}[rem]{Remark}
\newtheorem{ex}[ex]{Example}
\newtheorem*{fact*}{Fact}
\newtheorem*{rem*}{Remark}
\newtheorem*{ex*}{Example}
\aliascntresetthe{prop}
\aliascntresetthe{cor}
\aliascntresetthe{lem}
\aliascntresetthe{claim}
\aliascntresetthe{defn}
\aliascntresetthe{ques}
\aliascntresetthe{conj}
\aliascntresetthe{fact}
\aliascntresetthe{rem}
\aliascntresetthe{ex}

\usepackage[pointedenum]{paralist}
\usepackage{varioref}
\labelformat{equation}{\textnormal{(#1)}}
\labelformat{enumi}{\textnormal{(#1)}}

\usepackage[a4paper]{hyperref}

\def\textsectionN~{\textsection{}}

\renewcommand\phi{\varphi}
\renewcommand\epsilon{\varepsilon}
\renewcommand\leq{\leqslant}
\renewcommand\geq{\geqslant}

\makeatletter
\newcommand{\set}{  \@ifstar{\@setstar}{\@set}}\newcommand{\@setstar}[2]{\{\, #1 \mid #2 \,\}}
\newcommand{\@set}[1]{\{ #1 \}}
\makeatother

\newcommand{\trans}[1][1]{\raisebox{#1ex}{\scriptsize\kern0.1em$t$\kern-0.1em}}

\newcommand{\A}{\mathbb{A}}

\DeclareMathOperator{\Proj}{Proj}
\DeclareMathOperator{\Spec}{Spec}

\DeclareMathOperator{\Aut}{Aut}

\DeclareMathOperator{\vol}{vol}
\DeclareMathOperator{\Conv}{Conv}

\DeclareMathOperator{\ord}{ord}

\DeclareMathOperator{\DF}{DF}
\DeclareMathOperator{\Ding}{Ding}

\def\N{\mathbb{N}}
\def\Z{\mathbb{Z}}
\def\Q{\mathbb{Q}}
\def\R{\mathbb{R}}
\def\C{\mathbb{C}}
\def\A{\mathbb{A}}

\def\r+{\mathbb{R}_{\geq 0}}

\def\ep{\varepsilon}

\def\r+{{\R}_{\geq 0}}
\def\q+{{\Q}_{\geq 0}}
\def\P{\mathbb{P}}

\def\*c{\C^{\times}}

\def\tc{(\calx,\call)}
\def\tcL{(\calx_{\mathrm{Loe}},\call_{\mathrm{Loe}})}
\def\tcS{(\calx_{\mathrm{Soc}},\call_{\mathrm{Soc}})}

\def\lf{\calf^L_{\bullet} R}
\def\lfi{\calf^L_i R}
\def\lfik{\calf^L_i R_d}
\def\sf{\calg^S_{\bullet} R}
\def\sfi{\calg^S_i R}
\def\sfik{\calg^S_i R_d}
\def\fR{\calf_{\bullet} R}
\def\fiR{\calf_i R}

\def\mu{\mathfrak{u}}

\def\A{\mathbb {A}}

\def\C{\mathbb {C}}

\def\G{\mathbb {G}}

\def\N{\mathbb {N}}

\def\Q{\mathbb {Q}}
\def\R{\mathbb {R}}

\def\Z{\mathbb {Z}}

\newcommand{\calf}{\mathcal {F}}
\newcommand{\calg}{\mathcal {G}}

\newcommand{\call}{\mathcal {L}}

\newcommand{\calo}{\mathcal {O}}

\newcommand{\calx}{\mathcal {X}}

\makeatletter
  
  \@addtoreset{equation}{section}
\makeatother

\title[Examples on Loewy filtrations and K-stability of Fano varieties]{Examples on Loewy filtrations and K-stability of Fano varieties with non-reductive
automorphism groups}

\author[A.~Ito]{Atsushi~Ito}
\address{Graduate School of Mathematics,
Nagoya University,
Nagoya, Japan}
\email{atsushi.ito@math.nagoya-u.ac.jp}


\begin{document}

\maketitle

\begin{abstract}
It is known that the automorphism group of a K-polystable Fano manifold is
reductive. Codogni and Dervan construct a canonical filtration of the section ring, called
Loewy filtration, and conjecture that the Loewy filtration destabilizes any Fano variety with non-reductive automorphism group. In this note, we give a counterexample to their conjecture.
\end{abstract}

\section{Introduction}

For a Fano manifold $X$ over $\C$,
it is known that $X$ admits K\"ahler-Einstein metrics if and only if
$X$ is K-polystable \cite{Ti1,Do2,CT,St,Be,CDS1,CDS2,CDS3,Ti2}.
The K-polystability of $X$ is defined by using the Donaldson-Futaki invariant $\DF \tc$ of a test configuration $\tc$ of $X$.
Roughly,
$X$ is called \emph{K-polystable} if $\DF \tc \geq 0$ for any test configuration of $X$,
and equality holds only for a special type of test configurations, called \emph{of product type}.
On the other hand, Matsushima \cite{Ma} shows that if
$X$ admits K\"ahler-Einstein metrics then the automorphism group $\Aut(X)$ of $X$ is reductive.
Hence if $\Aut(X)$ is not reductive,
$X$ is not K-polystable.
Then
there exists a test configuration $\tc$ of $X$ which \emph{destabilizes} $X$,
i.e.\ $\DF \tc < 0$, or $\DF \tc =0$ and $\tc$ is not of product type.

By this observation,
Codogni and Dervan \cite{CD1} consider the following question:

\begin{ques}
If $\Aut(X)$ is not reductive,
can we find a (canonical) destabilizing test configuration $\tc$ of $X$ related to $\Aut(X)$?
\end{ques}

A test configuration $\tc$ can be interpreted as a suitable finitely generated decreasing filtrations 
$ \fR = \{\fiR\}_{ i \in \Z} $ of the section ring $R=\bigoplus_{d \geq 0} H^0(X,-dK_X)$
by
\[
(\calx,\call) =\left(\Proj_{\A^1} \bigoplus_i (\calf_i R ) t^{-i} , \calo(1) \right) \rightarrow \A^1=\Spec \C[t],
\]
where $ \fR  $ is called \emph{finitely generated} if  $\bigoplus_{i} (\fiR) t^{-i}$ is a finitely generated $\C[t]$-algebra.

 \vspace{2mm}
Using the action of $\Aut(X)$,
Codogni and Dervan construct a canonical filtration of $R$, called the \emph{Loewy filtration} of $X$.
Note that we do not know whether or not the Loewy filtration is finitely generated in general \cite{CD1,CD2}.

The following is a special case of \cite[Conjecture B]{CD1},
i.e.\ the case when
the Loewy filtration is finitely generated:

\begin{conj}\label{conj_main}
Let $X$ be a Fano manifold with non-reductive automorphism group.
Assume that the Loewy filtration
of $X$ is finitely generated.
Then the induced test configuration $\tcL$ destabilizes $X$.
\end{conj}

We note that they state the conjecture \cite[Conjecture B]{CD1} not only for Fano manifolds but also for polarized varieties.

The purpose of this note is to give a counterexample to \autoref{conj_main},
and hence to \cite[Conjecture B]{CD1} as follows:

\begin{thm}\label{thm_main}
There exists a smooth toric Fano $3$-fold $X$ with non-reductive automorphism group
such that the Loewy filtration is finitely generated and the Donaldson-Futaki invariant $\DF \tcL $ is positive.
In particular,
$\tcL$ does not destabilize $X$.
\end{thm}

In a preliminary version \cite{CD0} of \cite{CD1},
they also mention \emph{Socle filtrations},
which are ``dual'' of Loewy filtrations.
We also study Socle filtrations.

\vspace{2mm}
This note is organized as follows. 
In Section \ref{sec_Kstability}, we recall K-stability and Loewy filtrations.
In Section \ref{sec_toric}, we explain some known results about toric varieties.
In Section \ref{sec_ex}, we give a counterexample to \autoref{conj_main}.
In Appendix, we show a property of Socle filtrations.
Throughout this note, we work over $\C$.
We denote by $\N$ the set of all non-negative integers.

\subsection*{Acknowledgments}
The author would like to express his gratitude to
Professor Tomoyuki Hisamoto for suggesting me to consider Loewy filtrations and giving me useful comments and advice.
He is also grateful to Professor Naoto Yotsutani for valuable discussion and comments.
He would like to thank Professor Giulio Codogni for explaining stabilities and Loewy filtrations to him.
He is indebted to Professors Kento Fujita, Yuji Odaka, and Shintarou Yanagida for answering his questions. 
The author was supported by Grant-in-Aid for Scientific Research 17K14162.

\section{K-stability, test configurations, and filtrations}\label{sec_Kstability}

Throughout this section, $X$ is a $\Q$-Fano variety,
that is,
$X$ is a normal projective variety with at most klt singularities such that the anti-canonical divisor $-K_X$ is $\Q$-Cartier and ample.

\subsection{K-stability}

\begin{defn}\label{def_tc}
A \emph{test configuration} $(\calx,\call)$ of $X$ consists of the following data:
\begin{itemize}
\item a variety $\calx$ with a projective morphism $\pi:\calx \rightarrow \A^1$,
\item a $\Q$-line bundle $\call$ on $\calx$ which is ample over $\A^1$,
\item a $\G_m$-action on $(\calx,\call)$ such that $\pi$ is $\G_m$-equivariant and $(\calx \setminus X_0,\call|_{\calx \setminus X_0})$
is $\G_m$-equivariantly isomorphic to $(X \times (\A^1 \setminus \{0\}) , p_1^* (-K_X))$,
where $\G_m $ acts on $\A^1$ multiplicatively and $X_0$ is the fiber over $0 \in \A^1$.
\end{itemize}
\end{defn}

For a test configuration $(\calx,\call)$ of $X$,
we can define a rational number DF$(\calx,\call)$,
called the \emph{Donaldson-Futaki invariant} of $(\calx,\call)$.
See \cite{Do} for the definition of DF$(\calx,\call)$.

\begin{defn}\label{def_stability}
A $\Q$-Fano variety $X$ is called 
\begin{enumerate}
\item \emph{K-semistable} if for any test configuration $(\calx,\call)$ of $X$,
we have DF$(\calx,\call) \geq 0$.
\item \emph{K-polystable} if $X$ is K-semistable and, 
if DF$(\calx,\call) = 0$ for a test configuration $(\calx,\call)$ of $X$,
then $\calx $ is isomorphic to a test configuration of product type outside a codimension two subset, i.e.\
$\calx$ is isomorphic to $X \times \A^1$ over $\A^1$ outside a codimension two subset.
\end{enumerate}
\end{defn}

Let 
\[
R=\bigoplus_{d \geq 0} R_d=\bigoplus_{d \geq 0} H^0(X,-d K_X) 
\]
be the section ring of $X$.
In this note,
a \emph{decreasing filtration} $\calf_{\bullet} R$ of $R$ is a sequence of vector subspaces 
\[
\cdots \supset \calf_i R \supset \calf_{i+1} R \supset \cdots
\]
of $R$ for $i \in \Z$ such that 
$\calf_i R = \bigoplus_{d \geq 0} (\calf_i R \cap R_d)$ holds for any $i \in \Z$
and $\bigcup_{i \in \Z} \calf_i R=R$.

A decreasing filtration $\calf_{\bullet} R $ is called
\begin{itemize}
\item \emph{multiplicative} if $\calf_i R \cdot \calf_{j} R \subset \calf_{i+j} R$ for any $i,j$.
We note that if $\calf_{\bullet} R $ is multiplicative, $\bigoplus_{i \in \Z} (\calf_i R ) t^{-i}$ has a natural $\C[t]$-algebra structure.
\item \emph{finitely generated} if it is multiplicative and the $\C[t]$-algebra $\bigoplus_{i \in \Z} (\calf_i R ) t^{-i}$ is finitely generated.
\end{itemize}

An \emph{increasing filtration} $\calg_{\bullet} R$ of $R$ is a sequence of vector subspaces 
\[
\cdots \subset \calg_i R \subset \calg_{i+1} R \subset \cdots
\]
of $R$ for $i \in \Z$ such that $\calg_i R = \bigoplus_{d \geq 0} (\calg_i R \cap R_d)$ holds for any $i \in \Z$
and $ \bigcup_{i \in \Z}  \calg_i R = R$.

An increasing filtration $\calg_{\bullet} R $ is called
\begin{itemize}
\item \emph{multiplicative} if $\calg_i R \cdot \calg_{j} R \subset \calg_{i+j} R$ for any $i,j$.
In that case, $\bigoplus_{i \in \Z} (\calg_i R ) t^{i}$ has a natural $\C[t]$-algebra structure.
\item \emph{finitely generated} if it is multiplicative and the $\C[t]$-algebra $\bigoplus_{i \in \Z} (\calg_i R ) t^{i}$ is finitely generated.
\end{itemize}

If a decreasing filtration $\calf_{\bullet} R $ is finitely generated,
we have  
\[
(\calx,\call) :=\left(\Proj_{\A^1} \bigoplus_{i \in \Z} (\calf_i R ) t^{-i} , \calo(1) \right) \rightarrow \A^1,
\]
which is a test configuration of $X$.
We call this  $(\calx,\call)$ the test configuration induced by the finitely generated filtration $\calf_{\bullet} R $.
We say that $\calf_{\bullet} R $ \emph{destabilizes} $X$ if so does the induced test configuration $(\calx,\call) $.

Similarly, if an increasing filtration $\calg_{\bullet} R $ is finitely generated,
we have the induced test configuration
\[
(\calx,\call) :=\left(\Proj_{\A^1} \bigoplus_{i \in \Z} (\calg_i R ) t^{i} , \calo(1) \right) \rightarrow \A^1.
\]

\subsection{Loewy and Socle filtrations}

\begin{defn}\label{def_loewy_socle}
Let $U$ be a unipotent algebraic group,
and $V$ be a finite dimensional $U$-module.
\begin{enumerate}
\item The \emph{Loewy filtration} $\calf^L_{\bullet} V=\{ \calf^L_i V\}_{i \in \N}$ is a decreasing filtration of $U$-modules defined by
\begin{itemize}
\item[(i)] $\calf^L_0 V=V$,
\item[(ii)] for $i > 0$, $\calf^L_i V$ is the minimal $U$-submodule of $\calf^L_{i-1} V$ such that the quotient $\calf^L_{i-1} V/ \calf^L_{i} V $ is semisimple, i.e.\ the action on $\calf^L_{i-1} V/ \calf^L_{i} V $ is trivial.
\end{itemize}
\item The \emph{Socle filtration} $\calg^S_{\bullet} V = \{ \calg^S_i V\}_{i \in \N}$ is an increasing filtration of $U$-modules defined by
\begin{itemize}
\item[(i)] $\calg^S_0 V=V^U$, the invariant part of $V$ by the action of $U$,
\item[(ii)] for $i > 0$, $\calg^S_i V / \calg^S_{i-1} V = (V / \calg^S_{i-1} V )^U$.
\end{itemize}
\end{enumerate}
\end{defn}

\begin{rem}\label{rem_def_loewy_socle}
Loewy filtrations can be defined for not necessarily unipotent algebraic groups.
However, we can reduce the general case to the unipotent case by taking the unipotent radical \cite[Lemma 2.3]{CD1}.

Since $U$ is unipotent and $V$ is finite dimensional,
$\calf^L_iV=\{0\}$ and $\calg^S_i V=V$ for $i \gg 0$.

We also note that the indexes of the Socle filtration in \autoref{def_loewy_socle} is shifted by one from those in \cite{CD0}.
More precisely, it is defined as $\calg^S_0V =\{0\}, \calg^S_1 V=V^U, \dots$ in \cite{CD0}.
\end{rem}

\begin{ex}\label{ex_filtration_vecter_sp}
Fix $N \in \N$ and set $V_N =\{ f \in \C[x] \, | \, \deg(f) \leq N \} \subset \C[x]$.
Let $U $ be the additive unipotent algebraic group $\C$,
and consider the action of $U$ on $V_N$ by $\alpha \cdot x := x+\alpha$ for $\alpha \in U=\C$. 
In this case, it holds that
\[
\calf^L_i V_N=  \{ f \in V \, | \, \deg(f) \leq N-i \} , \quad \calg^S_i V_N=  \{ f \in V \, | \, \deg(f) \leq i \} .
\]
\end{ex}

\begin{defn}\label{def_loewy_socle_Section_ring}
Let $X$ be a $\Q$-Fano variety and
$U$ be the unipotent radical of the automorphism group $\Aut(X)$ of $X$.
Then $U$ acts on $R_d =  H^0(X,-d K_X) $ for each $d\geq0$.

\begin{enumerate}
\item The \emph{Loewy filtration} $\calf^{L}_{\bullet} R$ of $X$ is a decreasing filtration of $R$
defined by
\begin{itemize}
\item $\lfi =R$ for $i < 0$,
\item  $\lfi := \bigoplus_{d \geq 0} \calf^L_i R_d$ for $i \geq 0$,
where $\calf^L_{\bullet} R_d $ is the Loewy filtration of the $U$-module $R_d$.
\end{itemize}
\item The \emph{Socle filtration} $\calg^S_{\bullet} R$ of $X$ is an increasing filtration of $R$
defined by
\begin{itemize}
\item $\sfi =\{0\}$ for $i < 0$,
\item  $\sfi := \bigoplus_{d \geq 0} \calg^S_i R_d$ for $i \geq 0$,
where $\calg^S_{\bullet} R_d $ is the  Socle filtration of the $U$-module $R_d$.
\end{itemize}
\item If the Loewy filtration $\lf$ (resp.\ Socle filtration $\sf$) is finitely generated,
we denote by $\tcL$ (resp.\ $\tcS$) the induced test configuration of $X$.
\end{enumerate}
\end{defn}

\begin{rem}\label{rem_loewy_socle_Section_ring}
We note that $ \bigcup_{i \in \Z}  \calf^L_i R  =  \bigcup_{i \in \Z}  \calg^S_i R = R$ holds by \autoref{rem_def_loewy_socle}.
It is not known whether or not the Loewy filtration of a $\Q$-Fano variety is multiplicative in general \cite{CD2}.
On the other hand, we will show that the Socle filtration is multiplicative in Appendix.
\end{rem}

\begin{ex}
Let $S \rightarrow \P^2$ be the blow-up of $\P^2=\Proj \C[X,Y,Z]$ at $[1:0:0]$ and $[0:1:0]$.
The Loewy filtration of $S$ is computed in \cite[Subsection 3.2]{CD1} as follows.

The unipotent radical of $\Aut(S)$ consists of matrixes of the form
\[
 \begin{pmatrix}
1 & 0 & \alpha \\
0 & 1 & \beta\\
0&0&1\\
\end{pmatrix}
\quad \text{ for   }  \ \alpha, \beta \in \C,
\]
which acts on $\C[X,Y,Z] $ by
\begin{align}\label{eq_action}
X \mapsto X + \alpha Z, \quad Y \mapsto Y + \beta Z, \quad Z \mapsto Z.
\end{align}

Since $-K_S= 3H -E_1-E_2$, where $H$ is the pullback of $\calo_{\P^2}(1)$ and $E_1,E_2$ are the exceptional divisors,
we have
\[
R_d=H^0(S, -d K_S) = \langle X^a Y^b Z^{3d-a-b}  \, | \, 0 \leq a,b \leq 2d, a+b \leq 3d \rangle .
\]
In \cite{CD1}, it is shown that
$\calf^L_i R_d = \langle X^a Y^b Z^{3d-a-b}  \, | \, 0 \leq a,b \leq 2d, a+b \leq 3d -i \rangle$  for $i \geq 0$,
and hence $\lf$ is finitely generated.
In this example, $\DF \tcL < 0$ holds as computed in \cite{CD1}.

\vspace{2mm}
For the Socle filtration  $\calg^S_{\bullet} R$,
we need to compute the invariant part of the action of $U$.
By \ref{eq_action}, an element in $R_d=\langle X^a Y^b Z^{3d-a-b}  \, | \, 0 \leq a,b \leq 2d, a+b \leq 3d \rangle $ is invariant if and only if it a polynomial of $Z$.
Hence we have 
\[
\calg^S_0 R_d=R_d^{U} =\langle Z^{3d} \rangle.
\]
For $\calg^S_1 R_d $, we need to consider the action on 
\[
 R_d/\calg^S_0 R_d =  \langle X^a Y^b Z^{3d-a-b}  \, | \, 0 \leq a,b \leq 2d, a+b \leq 3d \rangle / \langle Z^{3d} \rangle.
\]
Since $  (R_d/\calg^S_0 R_d)^U =  \langle XZ^{3d-1}, YZ^{3d-1},Z^{3d}  \rangle / \langle Z^{3d} \rangle$,
we have $\calg^S_1 R_d = \langle XZ^{3d-1}, YZ^{3d-1},Z^{3d}  \rangle$.
Inductively,
it holds that
\[
\calg^S_{i} R_d = \langle X^a Y^b Z^{3d-a-b}  \, | \, 0 \leq a,b \leq 2d, a+b \leq \min \{i, 3d\} \rangle  
\]

In this example, 
the Socle filtration is essentially the same as the Loewy filtration.
More precisely, $\calg^S_{i} R_d = \calf^L_{3d-i} R_d $ holds for any $i, d$ and hence $\tcS$ coincides with $\tcL$.
\end{ex}

\section{Toric varieties}\label{sec_toric}

Let $M \simeq \Z^n$ be a lattice of rank $n$, and $N$ be the dual lattice of $M$.
An $n$-dimensional lattice polytope $P \subset M_{\R}:=M \otimes \R$ is called \emph{reflexive} if $P$ contains $0\in M$ in its interior 
and the dual polytope 
\[
P^{*} :=\{ v \in N_{\R} :=N \otimes \R \, | \, \langle u,v \rangle  \geq -1 \text{ for any } u \in P\}
\]
is a lattice polytope as well.
A reflexive polytope $P \subset M_{\R}$ defines
an $n$-dimensional Gorenstein toric Fano variety $X$ 
by 
\[
(X,-K_X)=\left(\Proj \C[\Gamma_P],\calo(1) \right),
\]
where $\Gamma_P=\{ (d,u) \in \N \times M \, | \, u \in dP\}$
and $\C[\Gamma_P] = \bigoplus_{(d,u) \in \Gamma_P} \C \chi^{(d,u)}$ is the semigroup ring graded by $\N$.
In particular, it holds that 
\[
H^0(X,-dK_X) =  \bigoplus_{u \in dP \cap M} \C \chi^{(d,u)}.
\]

In the rest of this section, $P \subset M_{\R}$ is a reflexive polytope and $X$ is the corresponding Gorenstein toric Fano variety.

\subsection{Toric test configurations}

Let $f : P \rightarrow \R$ be a piecewise linear concave function with rational coefficients.
As is well known, $f$ induces a test configuration of $X$ as follows.

Consider a decreasing filtration $\calf^f_{\bullet} R$ of the section ring $R=\C[\Gamma_P] $
by
\[
\calf^f_i R = \langle \chi^{(d,u)} \, | \,(d,u) \in \Gamma_P ,  f(u/d) \geq i/d \rangle.
\]
This filtration $\calf^f_{\bullet} R$ is multiplicative by the concavity of $f$, and finitely generated since $f$ is piecewise linear with rational coefficients.
Hence $\calf^f_{\bullet} R$ induces a test configuration $(\calx_f,\call_f)$ of $X$.

Similarly,
a piecewise linear convex function $g : P \rightarrow \R$ with rational coefficients
induces a finitely generated increasing filtration $\calg^g_{\bullet} R$ by
\[
\calg^g_i R = \langle \chi^{(d,u)} \, | \,(d,u) \in \Gamma_P ,  g(u/d) \leq i/d \rangle.
\]
In particular,
$\calg^g_{\bullet} R$ induces a test configuration $(\calx_g,\call_g)$ of $X$.

We note that $(\calx_f,\call_f) =(\calx_g,\call_g)$ holds if $g=C-f$ for some rational number $C$.

\vspace{2mm}

Other than the Donaldson-Futaki invariant $\DF(\calx,\call)$,
there exists another invariant $\Ding(\calx,\call)$ introduced in \cite{Be}, called the \emph{Ding invariant} of $(\calx,\call)$,
which also can be used to define K-stability.

For toric test configurations,
the following formulas are known:

\begin{thm}[\cite{Do},{\cite[Theorem 5, Proposition 7]{Ya}}]\label{thm_DFtoric}
Under the above notation,
it holds that
\begin{align*}
\DF(\calx_f,\call_f) &=n \left(\frac{1}{\vol(P)} \int_P f(u) du - \frac{1}{\vol(\partial P)} \int_{\partial P} f(u) d\sigma \right), \\
\Ding(\calx_f,\call_f) &= f(0) - \frac{1}{\vol(P)} \int_P f(u) du, \\
\DF(\calx_g,\call_g) &=n \left(-\frac{1}{\vol(P)} \int_P g(u) du + \frac{1}{\vol(\partial P)} \int_{\partial P} g(u) d\sigma \right), \\
\Ding(\calx_g,\call_g) &= - g(0) + \frac{1}{\vol(P)} \int_P g(u) du,
\end{align*}
where $du$ is the Euclidean measure on $M_{\R}$
and $d\sigma $ is the boundary measure on $\partial P$ induced by the lattice $M$.
The volumes $\vol(P),\vol(\partial P)$ are with respect to $du, d \sigma$ respectively.

Furthermore, it holds that 
\[
\DF(\calx_f,\call_f) \geq \Ding(\calx_f,\call_f) \quad  (\text{resp.} \, \DF(\calx_g,\call_g) \geq \Ding(\calx_g,\call_g)),
\]
and the equality holds if and only if $f$(resp.\ $g$) is radically affine,
where we say that a function $\phi:P \rightarrow \R$ is radically affine 
if $\phi(tu) - \phi(0) = t (\phi(u)-\phi(0))$ for any $t \in [0,1]$ and $u \in \partial P$.
\end{thm}

\subsection{Automorphism groups}\label{subsec_auto}

The automorphism group of toric varieties are studied by \cite{De,Co1,Co2,SMS}, etc.
For simplicity,
we only consider the Gorenstein Fano case here.


Let $v_1,\dots,v_N \in N_{\R}$ be all the vertices of $P^*$.
Then we have
\[
P=\{ u \in M_{\R} \, | \, \langle u,v_i \rangle \geq -1  \text{ for all } i\}.
\]
We denote by $D_i$ the torus invariant prime divisor on $X$ corresponding to $v_i$.

Let $S=\C[x_1,\dots,x_N]$, which is called the \emph{Cox ring} of $X$, be the polynomial ring whose variables correspond to the prime divisors $D_1,\dots,D_N$ on $X$.
Hence a torus invariant effective Weil divisor $D=\sum a_i D_i$ corresponds to
the monomial $x_1^{a_i} x_2^{a_2} \dots x_N^{a_N} \in S$, which is denoted by $x^D$.

Under this notation, $S$ is the direct sum of $\C x^D$ for all torus invariant effective Weil divisors $D$.
Hence the Cox ring is graded by the Chow group $A^1(X)$ of $X$ by
\[
S = \bigoplus_{\alpha  \in A^1(X) } S_{\alpha} = \bigoplus_{\alpha  \in A^1(X) } \left(\bigoplus_{ [D]=\alpha} \C x^D \right),
\]
where $[D]$ is the class of $D$ in $A^1(X)$.

\vspace{2mm}
We note that
the monomial $  \chi^{(d,u)} \in H^0(X,-d K_X)$ defines an effective torus invariant divisor $\sum_{i} (\langle u,v_i \rangle +d) D_i \in |-dK_X|$.
Thus we can naturally identify $H^0(X,-d K_X)$ with  $S_{ [-dK_X]} $.
Hence the section ring $R =\C[\Gamma_P]$ can be identified with the subring of $S$
\[
\bigoplus_{d \geq 0 } S_{ [-dK_X]}  \subset S.
\]

\begin{defn}\label{def_root}
An element $m \in M$ is called a \emph{root} of $P$ if there exists some $i$ such that
$ \langle m,v_i \rangle = -1$ and $\langle m,v_j \rangle \geq 0$ for any $j \neq i$. 
In other words, $m \in M$ is a root if and only if $m$ is contained in the relative interior of a facet $F$ of $P$.

A root $m$ is called \emph{semisimple} if $-m \in M$ is a root as well.
Otherwise, $m $ is called \emph{unipotent}.
\end{defn}

We note that
$-m$ is called a root in \cite{De,Co1} for a root $m$ in \autoref{def_root}.
We follow the notation in \cite{Ni}.

\begin{ex}
The reflexive polytope in \autoref{fig_root} has two semisimple roots $ \blacktriangle$ and two unipotent roots $ \bigstar$.

\begin{figure}[htbp]
  \begin{center}
    \includegraphics[scale=1]{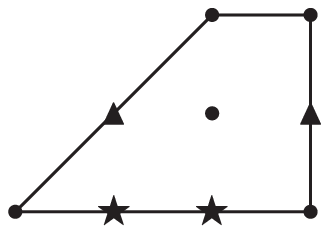}
    \caption{}
    \label{fig_root}
  \end{center}
\end{figure}

\end{ex}

\vspace{2mm}
For each root $m \in M$, we have a corresponding one-parameter subgroup $y_m : \C \rightarrow \Aut(X)$,
and the unipotent radical $U$ of $\Aut(X) $ is generated by $\bigcup_{m} y_m(\C)$,
where we take the union over all the unipotent roots of $P$.

Recall the definition of $y_m$.
Let $i$ be the unique index with $  \langle m,v_i \rangle = -1$ as in \autoref{def_root}.
We note that $D_i $ is linearly equivalent to the effective Weil divisor
$D =\sum_{j \neq i} \langle m,v_j \rangle D_j$.
For each $\alpha \in \C$, we have an automorphism of $S$ defined by
\[
x_i \mapsto x_i + \alpha x^{D}, \quad x_j \mapsto x_j \  \text{ for }  \ j \neq i ,
\]
which preserves the $A^1(X)$-grading.
This induces an automorphism 
$y_m(\alpha) \in \Aut(X)$.


\section{Examples}\label{sec_ex}

Let $P \subset M_{\R} $ be a reflexive polytope and $X$ be the corresponding Gorenstein toric Fano variety.
In this section,
we only consider examples with the simplest non-trivial unipotent radical,
that is,
we assume that there exists a unique unipotent root $m$ of $P$ throughout this section.
Hence the unipotent radical $U$ of $\Aut(X)$ is isomorphic to $\C$ via the one-parameter subgroup $y_m$. 
In this case, 
the Loewy and Socle filtrations and the Donaldson-Futaki invariants of them are described as follows.

\vspace{2mm}
Let $F$ be the unique facet of $P$ containing $m$.
Without loss of generality,
we may assume $ M = M' \times \Z$ for $M' \simeq \Z^{n-1}$, $m=(0,-1) \in  M' \times \Z$, and $F=F' \times \{-1\}$
for a lattice polytope $F' \subset {M'}_{\R}$.
By \cite[Lemma 5.9]{Ni},
there exists a piecewise linear concave function $h : F' \rightarrow  \R$ such that
\begin{align}\label{eq_P}
P=\{ (u',t) \in F' \times \R \, | \,  -1 \leq t \leq h(u') \}.
\end{align}

\begin{ex}\label{ex_deg6}
For the reflexive polytope $P \subset \R^2$ in \autoref{fig:deg6},
$F'=[-1,1] \subset \R$ and $h(u')= 1-|u'|$.

\begin{figure}[htbp]
  \begin{center}
    \includegraphics[]{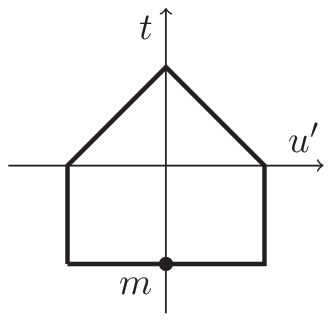}
    \caption{}
    \label{fig:deg6}
  \end{center}
\end{figure}

\end{ex}

For $u' \in dF' \cap M'$,
set 
\begin{align*}
R_d^{u'} &= \langle \chi^{(d,u)} \, | \, u=(u',l) \in dP \cap M \rangle \\
&= \langle \chi^{(d,u)} \, | \, u=(u',l)  \text{ with } l \in \Z, -d \leq l \leq \lfloor d h(u'/d) \rfloor \rangle  \subset R_d .
\end{align*}
By \ref{eq_P},
we have a decomposition 
\[
R_d=\bigoplus_{u' \in dF' \cap M'} R_d^{u'}  
\]
 as vector spaces.
In fact, this is a decomposition as $U$-modules by the following lemma:

\begin{lem}\label{lem_isom_Vd}
For any $u' \in dF' \cap M'$,
$R_d^{u'}$ is a $U$-submodule of $R_d$.
Furthermore $R_d^{u'}$ is isomorphic to $V_{ \lfloor dh(u'/d) \rfloor+d}$ in \autoref{ex_filtration_vecter_sp} as $U$-modules.
\end{lem}

\begin{proof}
As in \autoref{subsec_auto}, let $D_1,\dots,D_N$ be all the torus invariant prime divisors on $X$.
We may assume that the facet $F \ni m=(0,-1)$ corresponds to $D_1$.
Recall that $\alpha \in \C=U \subset \Aut(X)$ acts on the Cox ring $S$ by
\[
x_1 \mapsto x_1 + \alpha x^{D}, \quad x_i \mapsto x_i \ \text{ for } \ i \geq 2 ,
\]
where $D=\sum_{i \geq 2} \langle m, v_i \rangle D_i$.

Fix $u' \in dF' \cap M'$.
For simplicity, set $\chi_l=\chi^{(d,u)} \in R_d^{u'} $ for $u=(u',l)$ with $-d \leq l \leq \lfloor d h(u'/d) \rfloor$.
By the identification of $R=\C[\Gamma_P]$ with $\bigoplus_{d\geq 0} S_{[-dK_X]}$ in \autoref{subsec_auto},
$\chi_{l}  \in R_d^{u'}  \subset R$ is identified with
\[
X_l:= \prod_{i=1}^N x_i^{\langle u, v_i \rangle +d}  \in S.
\]
Since $v_1 = (0,1) \in N' \times \Z$, where $N' $ is the dual lattice of $M'$,
we have $\langle u, v_1 \rangle + d =l+d$.
Hence 
\[
X_l= x_1^{l+d} \prod_{i=2}^N x_i^{\langle u, v_i \rangle +d}  \in S,
\]
which is mapped to 
\[
(x_1+\alpha x^D)^{l+d} \prod_{i=2}^N x_i^{\langle u, v_i \rangle +d}   
\]
by the action of $\alpha \in \C$.
Since $x^D=  \prod_{i=2}^N x_i^{\langle m, v_i \rangle} $,
\begin{align*}
(x_1+\alpha x^D)^{l+d} &= \sum_{j=0}^{l+d}  \binom{l+d}{j} \alpha^{j} x_1^{l+d-j}   x^{jD}  \\
&=   \sum_{j=0}^{l+d}  \binom{l+d}{j} \alpha^{j} x_1^{l+d-j}   \prod_{i=2}^N x_i^{j\langle m, v_i \rangle} .
\end{align*}
Thus by the action of $\alpha \in \C$, $X_l$ is mapped  to
\begin{align*}
(x_1+\alpha x^D)^{l+d} \prod_{i=2}^N x_i^{\langle u, v_i \rangle +d} 
&= \left( \sum_{j=0}^{l+d}  \binom{l+d}{j} \alpha^{j} x_1^{l+d-j}   \prod_{i=2}^N x_i^{j\langle m, v_i \rangle} \right)  \prod_{i=2}^N x_i^{\langle u, v_i \rangle +d} \\
&=\sum_{j=0}^{l+d}  \binom{l+d}{j} \alpha^{j} x_1^{l+d-j}   \prod_{i=2}^N x_i^{\langle u+jm, v_i \rangle +d}\\
&=\sum_{j=0}^{l+d}  \binom{l+d}{j} \alpha^{j} X_{l-j},
\end{align*}
where the last equality follows from $ u+jm=(u',l) +j(0,-1)=(u',l-j)$.
In particular,
$\langle X_l \, | \, -d \leq l \leq  \lfloor dh(u'/d) \rfloor \rangle \subset S$ is closed under the action of $U$.
Hence so is $R_d^{u'} =\langle \chi_l \, | \, -d \leq l \leq  \lfloor dh(u'/d) \rfloor \rangle \subset R_d $,
i.e.\
$R_d^{u'} $ is a $U$-submodule.

By the above argument,
\begin{align}\label{eq_isom}
R_d^{u'}  \rightarrow V_{ \lfloor dh(u'/d) \rfloor +d} \quad : \quad  \chi_l \mapsto x^{l+d}
\end{align}
is an isomorphism as $U$-modules.
\end{proof}

\begin{lem}\label{lem_loewy_socle}
Under the above setting, for $u =(u',l) \in d P \cap M$,
$\chi^{(d,u)} \in R_d$ is contained in $\lfik$ if and only if
$l \leq dh(u'/d) -i$. 

On the other hand,
$\chi^{(d,u)} \in R_d$ is contained in $\sfik$ if and only if
$l \leq i-d $.
\end{lem}

\begin{proof}
By \ref{eq_P}, this lemma holds for $i <0$.
Hence we may assume $i \geq 0$.

We use the notation in the proof of \autoref{lem_isom_Vd}.
By \autoref{lem_isom_Vd},
we have a decomposition $ R_d=\bigoplus_{u' \in dF' \cap M'} R_d^{u'}  $ as $U$-modules.
Hence $\lfik=\bigoplus_{u' \in dF' \cap M'} \calf^L_i R_d^{u'}$ holds.

Since $\calf^L_i V_{ \lfloor dh(u'/d) \rfloor +d} = \langle x^j \, | \, 0 \leq j \leq  \lfloor dh(u'/d)  \rfloor  + d-i \rangle$ by \autoref{ex_filtration_vecter_sp},
we have 
\[
\calf^L_i R_d^{u'} = \langle \chi_l  \, | \, -d \leq l \leq  \lfloor dh(u'/d) \rfloor -i \rangle
\]
by \ref{eq_isom}.
Thus $\chi_l= \chi^{(d,u)}$ for $u =(u',l) $ is contained in $ \calf^L_i R_d  $ if and only if 
$l \leq  \lfloor dh(u'/d) \rfloor -i $,
which is equivalent to $ l \leq   dh(u'/d) -i $  since $l$ and $i$ are integers.

Similarly,
we have $\sfik=\bigoplus_{u' \in dF' \cap M'} \calg^S_i R_d^{u'}$ and
\[
\calg^S_i R_d^{u'} = \langle \chi_l  \, | \, -d \leq l \leq -d+ i \rangle.
\]
Hence $\chi_l= \chi^{(d,u)}$ is contained in $ \calg^S_i R_d  $ if and only if $l \leq -d+i$.
\end{proof}

\begin{prop}\label{prop_filtration}
Under the above setting,
the Loewy filtration $\lf$ of $X$ coincides with
the decreasing toric filtration $\calf^f_{\bullet} R$ induced by the concave function $f$ defined as
\[
f: P \rightarrow \R, \quad (u',t) \mapsto h(u')-t.
\]

On the other hand,
the Socle filtration $\sf$ of $X$ coincides with
the increasing toric filtration $\calg^g_{\bullet} R$ induced by the affine (hence convex) function $g$ defined as
\[
g: P \rightarrow \R, \quad (u',t) \mapsto t+1.
\]
\end{prop}

\begin{proof}
By the definition of toric filtrations, for $u =(u',l) \in d P \cap M$,
$\chi^{(d,u)} \in R_d$
is contained in $\calf^f_i R_d$ if and only if 
\[
i/d \leq f(u/d) =f(u'/d,l/d)= h(u'/d) - l/d,
\]
which is equivalent to $l \leq dh(u'/d) - i$.
Hence $\calf^f_{\bullet} R$ coincides with the Loewy filtration $\lf$ by \autoref{lem_loewy_socle}.

Similarly,
$\chi^{(d,u)} \in R_d$
is contained in $\calg^g_i R_d$ if and only if 
\[
i/d \geq g(u/d) = g(u'/d,l/d)=l/d+1,
\]
which is equivalent to $ l \leq i-d$.
Hence $\calg^g_{\bullet} R$ coincides with the Socle filtration $\sf$ by \autoref{lem_loewy_socle}.
\end{proof}

Since $P=\{ (u',t) \in F' \times \R \, | \,  -1 \leq t \leq h(u') \}$,
roughly  \autoref{prop_filtration} states
that the Loewy (resp.\ Socle) filtration is determined by the distance from the top facets of $P$ defined by $h$ 
(resp.\ the distance from the bottom facet $F=F'\times \{-1\}$).

\vspace{2mm}
By \autoref{prop_filtration},
both $\lf$ and $\sf$ are finitely generated, and hence induce test configurations $\tcL$ and $\tcS$, respectively.
By \autoref{thm_DFtoric},
we can compute the Donaldson-Futaki invariant and the Ding invariant of these test configurations as follows:

\begin{cor}\label{cor_invariants}
It holds that
\begin{align*}
\Ding \tcL &=  \frac{1}{\vol(P)} \int_P (h(0) - h(u') +t ) du'dt, \\
\DF \tcS =\Ding \tcS &=  \frac{1}{\vol(P)} \int_P t \, du'dt  .
\end{align*}
If $h : F' \rightarrow \R$ is radically affine, $\DF \tcL =\Ding \tcL $ holds.
\end{cor}

\begin{proof}
This follows from \autoref{thm_DFtoric} and \autoref{prop_filtration}.
We note that $g(u',t)=t+1$ is affine, in particular, radically affine.
On the other hand, $f(u',t) = h(u') -t$ is radically affine if and only if so is $h$.
\end{proof}

In all the following examples,
$h$ is radically affine and hence $\DF \tcL = \Ding \tcL $ holds.

\subsection{A singular toric del Pezzo surface}\label{subsec_surface}

In this subsection, we give a counterexample to \autoref{conj_main}
with singular $X$.

\vspace{2mm}
Let $P \subset \R^2$ be the reflexive polytope in \autoref{fig:deg6}.
We note that the corresponding $X$ is a singular del Pezzo surface of degree $6$ with one ordinary double point. 
In this case,
$F' = [-1,1] $ 
and $h:F' \rightarrow \R$ is defined by $h(x) = 1-|x|$ as stated in \autoref{ex_deg6}.
Since $h$ is radically affine, 
we have
\begin{align*}
\DF \tcL &=  \frac{1}{\vol(P)} \int_P (|x| + t ) dxdt = \frac29 >0, \\
\DF \tcS &=  \frac{1}{\vol(P)} \int_P t \, dxdt =- \frac29 < 0 .
\end{align*}
by \autoref{cor_invariants}.
Hence the Loewy filtration does not destabilize $X$, but the Socle filtration does.

\subsection{A smooth toric Fano $3$-fold}\label{subsec_3fold}

In this subsection, 
we show \autoref{thm_main}, i.e.\
we give a counterexample to \autoref{conj_main} with smooth $X$.

The reflexive polytope 
\[
F' =\Conv((1,1), (0,1), (-2,-1),(1,-1)) \subset \R^2
\]
in \autoref{fig:sm_3fold} corresponds to the Hirzebruch surface $\Sigma_1=\P_{\P^1}(\calo \oplus \calo(-1))$.
Let $X$ be the smooth toric Fano $3$-fold obtained as the blow-up of $\Sigma_1 \times \P^1$
along $C \times \{p\}$, where $C \subset \Sigma_1$ is the torus invariant section with $(C^2)=1$ and $p \in \P^1$ is a torus invariant point.
Since $\Sigma_1 \times \P^1$ corresponds to $F' \times [-1,1]$,
the polytope $P $ corresponding to $X$ is written as
\[
P=\left\{ (x,y,t) \in F' \times \R \subset \R^3 \, | \, -1 \leq t \leq h(x,y) :=\min \{ 1, 1+y\} \right\}.
\]
We note that $P$ has two semisimple roots and one unipotent root $m=(0,0,-1)$.

\begin{figure}[htbp]
  \begin{center}
    \includegraphics[]{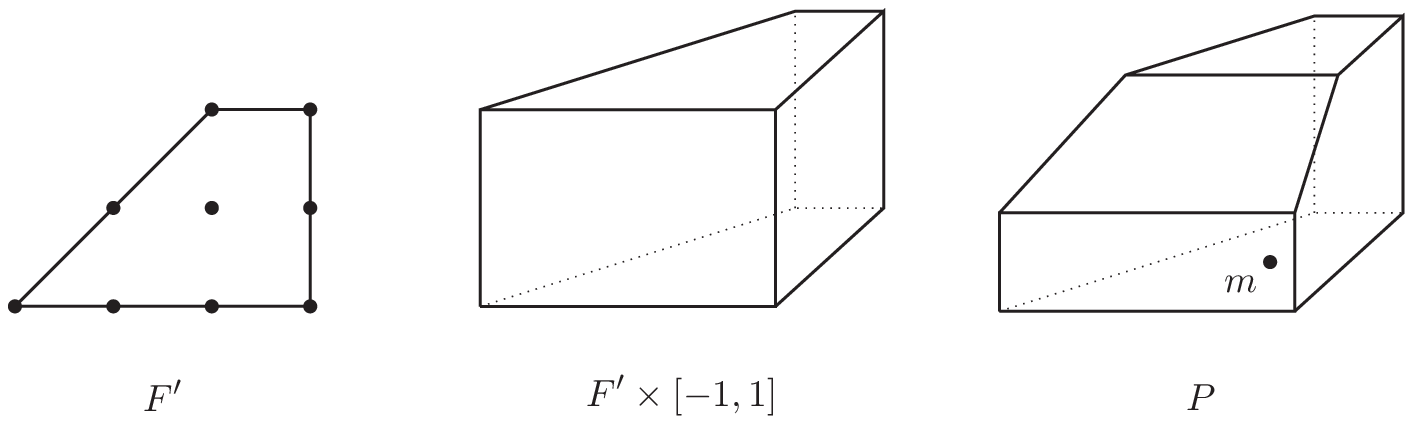}
    \caption{}
    \label{fig:sm_3fold}
  \end{center}
\end{figure}

Since $h$ is radically affine, 
we have
\begin{align*}
\DF \tcL &=  \frac{1}{\vol(P)} \int_P (\max \{ 0,-y\} + t ) dxdydt = \left(\frac{20}{3} \right)^{-1} \! \! \! \! \cdot \frac78 = \frac{21}{160}, \\
\DF \tcS &=  \frac{1}{\vol(P)} \int_P t \, dxdydt = \left(\frac{20}{3} \right)^{-1} \! \! \! \! \cdot \left(-\frac78\right) = -\frac{21}{160} .
\end{align*}

\begin{proof}[Proof of \autoref{thm_main}]
The above $X$ satisfies the conditions in the theorem.
\end{proof}

\subsection{A singular toric Fano $3$-fold}\label{subsec_sing3fold}

For examples in Subsections \ref{subsec_surface} and \ref{subsec_3fold},
the invariant $\DF \tcS$ is negative, and hence
the Socle filtration destabilizes $X$.

As we will see in Appendix,
the Socle filtration is the filtration induced from a valuation on the function field of $X$,
and hence multiplicative in general.
Thus we might expect that the Socle filtration destabilizes any $\Q$-Fano varieties.

However, the answer is no, at least for singular $X$.
The following is an example of a Gorenstein toric Fano $3$-fold with non-reductive automorphism group such that
$\DF \tcS = \Ding \tcS >0$.

\vspace{2mm}

Let $F' \subset \R^2$ be the hexagon with vertexes $(1,0),(0,1),(-1,1),(-1,0),(0,-1),(1,-1)$.
We define a function $h : F' \rightarrow \R$ by
\[
  h(x,y) = \left\{
    \begin{array}{ll}
      1-2x & (x \geq 0) \\
      1-x & (x \leq 0)
    \end{array}
  \right.
\]
for $(x,y) \in F'$.
The polytope $P \subset \R^3$ in \autoref{fig:3fold} is defined by $h$ and \ref{eq_P}.
We can check that $P$ is reflexive,
and $m=(0,0,-1) \in P$ is the unique unipotent root.
By \autoref{cor_invariants},
we can compute 
\begin{align*}
\DF \tcL &= \left(\frac{16}{3} \right)^{-1} \! \! \! \! \cdot \left(-\frac38\right) =- \frac{9}{128} < 0, \\
\DF \tcS & = \left(\frac{16}{3} \right)^{-1} \! \! \! \! \cdot \frac38 = \frac{9}{128} >0. 
\end{align*}

\begin{figure}[htbp]
  \begin{center}
    \includegraphics[]{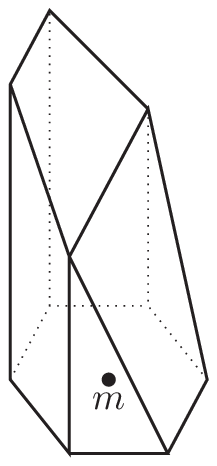}
    \caption{}
    \label{fig:3fold}
  \end{center}
\end{figure}


\appendix

\section{On Socle filtrations}\label{sec_app}

Let $R=\bigoplus_{d=0}^{\infty} R_d$ be a finitely generated graded integral $\C$-algebra
and set $X=\Proj R$.
We do not assume that $X $ is Fano.
Let $U$ be a unipotent algebraic group which acts on $R $ as a graded $\C$-algebra.
By exactly the same definition as \autoref{def_loewy_socle},
we can define the Socle filtration $\sf$ of $R$.

\vspace{2mm}

Recall that an increasing filtration $\calg_{\bullet} R$ is multiplicative if and only if 
$\calg_{i} R \cdot \calg_{j} R \subset \calg_{i+j} R$ holds for any $i,j$.

\begin{lem}\label{lem_multiplicative}
Under the above setting,
the Socle filtration $\sf$ is multiplicative.
\end{lem}

To show this lemma, we use the Lie algebra $ \mathfrak{u}$ of $U$.
Since $U$ acts on $R$ as a $\C$-algebra,
any $D \in \mu $ acts on $R$ as a $\C$-derivation,
i.e.\ $Dc=0 $ for any $c \in \C$ and 
\[
D(xy) = (Dx) y + x (Dy) 
\]
holds for any $x,y \in R$.
By induction, for any $ D_1,\dots,D_N \in \mu$ 
it holds that
\begin{align}\label{eq_DDDD(xy)}
D_{ N} \cdots D_1 (xy) =  \sum_{(\ep_1,\dots,\ep_{N})  \in \{0,1\}^{N}} (D_{N}^{\ep_{N}} \cdots D_1^{\ep_1} x)(D_{N}^{1-\ep_{N}} \cdots D_1^{1-\ep_1} y),
\end{align}
where $D^0 x=x$ by convention.

\begin{lem}\label{lem_socle_derivation}
For any $i, d \geq 0$,
it holds that
\[
\sfik = \{ x \in R_d \, | \, D_{i+1} \dots D_1 x =0 \text{ for any } D_1,\dots,D_{i+1} \in \mu \}.
\]
\end{lem}

\begin{proof}
For $i=0$,
$x \in R_d$ is contained in the invariant part $\calg^S_0 R_d= (R_d)^U$ if and only if  
$Dx=0$ for any $D \in \mu$.
Hence the statement holds for $i=0$.
By the induction on $i$,
this lemma follows.
\end{proof}

\begin{proof}[Proof of \autoref{lem_multiplicative}]
Take $x \in \calg_i^S R$ and $y \in  \calg_i^S R$ for $i,j \in \Z$.
We need to show $xy \in  \calg_{i+j}^S R$.
Since $\calg_k^S R =\{0\}$ for $k < 0$ by definition,
$xy=0 \in  \calg_{i+j}^S R$ holds if $i$ or $j$ is negative.
Hence we may assume $i,j \geq 0$.

Set $N=i+j+1$
and take any $  D_1,\dots,D_{N} \in \mu$.
It suffices to show $D_{N} \dots D_1 (xy) =0 $ by \autoref{lem_socle_derivation}.
For each $(\ep_1,\dots,\ep_{N}) \in \{0,1\}^{N}$,
one of $\sum\ep_k \geq i+1$ or $\sum (1-\ep_k) \geq j+1$ must hold.
Hence $D_{N}^{\ep_{N}} \cdots D_1^{\ep_1} x=0 $ or $D_{N}^{1-\ep_{N}} \cdots D_1^{1-\ep_1} y=0$  holds by \autoref{lem_socle_derivation}.
By \ref{eq_DDDD(xy)},
we have $D_{N} \dots D_1 (xy) =0 $.
\end{proof}

In fact,
we can show the following proposition,
which refines \autoref{lem_multiplicative}.

\begin{prop}\label{prop_app}
Let $x \in \calg_i^S R \setminus \calg_{i-1}^S R $ and $y \in  \calg_j^S R \setminus  \calg_{j-1}^S R $ for $i,j \geq 0$.
Then $xy \in \calg_{i+j}^S R \setminus \calg_{i+j-1}^S R $ holds.
\end{prop}

\begin{proof}
Since $xy \in \calg_{i+j}^S R$ by \autoref{lem_multiplicative},
what we need to show is $xy \not \in \calg_{i+j-1}^S R $.
By \autoref{lem_socle_derivation},
it is enough to find $ D_1,\dots,D_{i+j} \in \mu$ such that $D_{ i+j} \cdots D_1 (xy) \neq 0$.

Consider the set $\Phi$ which consists of sequences of non-negative integers
$(a_k)_{k=1}^{\infty}$ satisfying
\begin{itemize}
\item $\sum_{k=1}^{\infty} a_k=i$. 
In particular, there exists $m $ such that $a_k =0$ for any $k \geq m+1$.
\item For the above $m$,
there exist $D_1,\dots, D_m \in \mu$ such that 
$ D_{m}^{a_m} \cdots D_1^{a_1} x \neq 0$.
\end{itemize}
We note that $a_k$ could be zero even if $k \leq m$.
For simplicity,
we denote $(a_k)_{k=1}^{\infty}$ by $ (a_1,\dots,a_m) $ if $a_k =0$ for any $k \geq m+1$.

Since $x \not \in \calg_{i-1}^S R $,  $D_{i} \cdots  D_1 (x) \neq 0$ for some $D_1,\dots,D_{i} \in \mu$.
Hence $(\underbrace{1,1,\dots,1}_{i})$ is contained in $\Phi$.
In particular,
$\Phi \neq \emptyset$.

Let $(a_1,\dots,a_m) = (a_1,\dots,a_m,0,0,\dots)    \in \Phi$ be the maximum element with respect to the lexicographical order.
Take and fix $D_1,\dots,D_m \in \mu$ with $ D_m^{a_m} \cdots D_1^{a_1} x \neq 0 $.

\vspace{2mm}
Consider another set $\Phi' \subset \N^{m}$ defined as follows:
$(a'_1,\dots,a'_{m}) \in \N^{m}$ is contained in $\Phi' $ if and only if 
\begin{itemize}
\item $n:=j-  (a'_1 + \dots +a'_{m}) \geq 0$ and $ D'_n \dots D'_1 D_m^{a'_m} \cdots D_1^{a'_1} y \neq 0$ for some $D'_1,\dots, D'_n \in \mu$.
\end{itemize}
Since $y \not \in \calg_{j-1}^S R$,  $D'_j \cdots  D'_1 y \neq 0$ for some $D'_1,\dots,D'_j$.
Hence $(\underbrace{0,0,\dots,0}_{m})$ is contained in $\Phi'$.
In particular, $\Phi' \neq \emptyset$.

Let $(a'_1,\dots,a'_{m}) \in \Phi'$ be the maximum element with respect to the lexicographical order.
Take and fix $D'_1,\dots, D'_n \in \mu$ with $ D'_n \dots D'_1 D_m^{a'_m} \cdots D_1^{a'_1} y \neq 0$ for $n=j-  (a'_1 + \dots +a'_{m}) $.

\vspace{2mm}
To prove $xy \not \in \calg_{i+j-1}^S R $,
it suffices to show
\begin{align}
D'_n \cdots D'_1 D_m^{a_m+ a'_m} \cdots D_1^{a_1+a'_1} (xy) \neq 0
\end{align}
since $  \sum_{i=1}^m (a_i+a'_i) +n=\sum_{i=1}^m a_i + (n+  \sum_{i=1}^m a'_i)  =i+j$.

By \ref{eq_DDDD(xy)}, 
$D'_n \cdots D'_1 D_m^{a_m+ a'_m} \cdots D_1^{a_1+a'_1} (xy) $ is equal to
\begin{align}\label{eq_DDDDDxy}
\sum_{\boldsymbol{\alpha},\boldsymbol{\ep}} c_{\boldsymbol{\alpha},\boldsymbol{\ep}} ({D'_n}^{\ep_n} \cdots {D'_1}^{\ep_1} D_m^{\alpha_m} \cdots D_1^{\alpha_1} x )
 ({D'_n}^{1-\ep_n} \cdots {D'_1}^{1-\ep_1} D_m^{a_m+ a'_m-\alpha_m} \cdots D_1^{a_1+a'_1-\alpha_1  } y ) 
\end{align}
where the sum is taken over all $(\boldsymbol{\alpha},\boldsymbol{\ep}) =(\alpha_1,\dots,\alpha_m,\ep_1,\dots, \ep_n) $ with 
\[
\quad \alpha_i \in \{0,1,\dots, a_i+a'_i\} , \quad \boldsymbol{\ep}  \in \{0,1\}^n
\]
and the coefficient $c_{\boldsymbol{\alpha},\boldsymbol{\ep}} \in \N$ is
\[
c_{\boldsymbol{\alpha},\boldsymbol{\ep}} = \prod_{i=1}^m \binom{a_i+a'_i}{\alpha_i}.
\]
If $\sum_{i=1}^m \alpha_i  +\sum_{j=1}^n \ep_j  > i$,
it holds that  ${D'_n}^{\ep_n} \cdots {D'_1}^{\ep_1} D_m^{\alpha_m} \cdots D_1^{\alpha_1} x =0$ by $x \in \sfi$.
If  $\sum_{i=1}^m \alpha_i  +\sum_{j=1}^n \ep_j  <  i$, ${D'_n}^{1-\ep_n} \cdots {D'_1}^{1-\ep_1} D_m^{a_m+ a'_m-\alpha_m} \cdots D_1^{a_1+a'_1-\alpha_1  } y =0$
by $y \in \calg^S_j R$ and $ \sum_{i=1}^m (a_i+a'_i- \alpha_i) + \sum_{j=1}^n (1-\ep_j ) > j$.
Hence it suffices to take the sum in \ref{eq_DDDDDxy} over $(\boldsymbol{\alpha},\boldsymbol{\ep})$ with
\begin{align}\label{eq_sum=i}
\sum_{i=1}^m \alpha_i  +\sum_{j=1}^n \ep_j =i.
\end{align}

Assume that $(\boldsymbol{\alpha},\boldsymbol{\ep})$ satisfies \ref{eq_sum=i}.
By the definition of $\Phi$, 
${D'_n}^{\ep_n} \cdots {D'_1}^{\ep_1} D_m^{\alpha_m} \cdots D_1^{\alpha_1} x =0$
 if $ (\alpha_1,\dots,\alpha_m, \ep_1,\dots,\ep_n) \not \in \Phi$.
 Since $(a_1,\dots,a_m) \in \Phi$ is the maximum element,
 it suffices to take the sum in \ref{eq_DDDDDxy} over $(\boldsymbol{\alpha},\boldsymbol{\ep})$ with \ref{eq_sum=i} and
 \begin{align}\label{eq_S}
 (\alpha_1,\dots,\alpha_m, \ep_1,\dots,\ep_n)  \leq (a_1,\dots,a_m).
\end{align}

By the definition of $\Phi'$,
${D'_n}^{1-\ep_n} \cdots {D'_1}^{1-\ep_1} D_m^{a_m+ a'_m-\alpha_m} \cdots D_1^{a_1+a'_1-\alpha_1  } y=0$
 if $ (a_1+a'_1-\alpha_1,\dots,a_m+ a'_m-\alpha_m) \not \in \Phi'$.
Since $(a'_1,\dots,a'_{m})  \in \Phi'$ is the maximum element,
 it suffices to take the sum in \ref{eq_DDDDDxy} over $(\boldsymbol{\alpha},\boldsymbol{\ep})$ with \ref{eq_sum=i}, \ref{eq_S} and
 \begin{align}\label{eq_S'}
(a_1+a'_1-\alpha_1,\dots,a_m+ a'_m-\alpha_m) \leq (a'_1,\dots,a'_{m}) .
\end{align}

\vspace{2mm}
Assume that  the index $(\boldsymbol{\alpha},\boldsymbol{\ep})$  satisfies \ref{eq_sum=i}, \ref{eq_S}, and \ref{eq_S'}.
Then $\alpha_1 \leq a_1$ and $a_1+a'_1 - \alpha_1 \leq a'_1$ hold.
Hence $\alpha_1$ must be $a_1$.

Since $\alpha_1 = a_1$,
we have $\alpha_2 \leq a_2$ and $a_2+a'_2 - \alpha_2 \leq a'_2$,
which imply $\alpha_2=a_2$.
Repeating this,
$(\alpha_1,\dots,\alpha_m)$ must coincide with $(a_1,\dots,a_m)$.
By \ref{eq_sum=i} and $\sum_{i=1}^m a_i =i$,
we have $\boldsymbol{\ep} =(0,\dots,0)$.

After all,
the index $(\boldsymbol{\alpha},\boldsymbol{\ep})$ which we need to take is only  $((a_1,\dots,a_m) ,(0,\dots,0))$.
Hence $D'_n \cdots D'_1 D_m^{a_m+ a'_m} \cdots D_1^{a_1+a'_1} (xy) $ is equal to
\[
 c_{(a_1,\dots,a_m) ,(0,\dots,0) } (D_m^{a_m} \cdots D_1^{a_1} x ) ({D'_n} \dots D'_1 D_m^{a'_m} \cdots D_1^{a'_1} y),
\]
which is nonzero since both $  D_m^{a_m} \cdots D_1^{a_1} x $ and $  {D'_n} \dots D'_1 D_m^{a'_m} \cdots D_1^{a'_1} y $ are non-zero elements in the integral domain $R$,
and $c_{(a_1,\dots,a_m) ,(0,\dots,0) }  \neq 0$. 
Thus $xy \not \in \calg_{i+j-1}^S R$ follows.
\end{proof}

\autoref{prop_app} implies that the Socle filtration induces a valuation on the function field of $X$ as follows.

\begin{defn}\label{def_function}
For $x \in R$, we set
\begin{align}
\iota(x) = \inf \{ i \in \Z \, | \, x \in \sfi\} \in \{-\infty\} \cup \N.
\end{align}
We note that $ \{ i \in \Z \, | \, x \in \sfi\}  \neq \emptyset$ for any $x \in R$ since $\bigcup_{i} \sfi=R$,
and $\iota(x)=- \infty $ if and only if $x=0$, and $\iota(c)=0$ for $c \neq 0\in \C \subset R_0$.
For $x,y \in R$, $\iota(xy)=\iota(x) + \iota(y)$ holds by \autoref{prop_app}.
\end{defn}


\begin{defn}\label{def_v}
Let $K(X)$ be the function field of $X$.
We define a function $v : K(X) \rightarrow \Z \cup \{\infty\}$ by
\[
v\left(\frac{x}{y}\right) = -\iota(x) + \iota(y)
\]
for $d \geq 0, x,y \in R_d, y \neq 0$.
\end{defn}

\begin{cor}\label{cor_valuation}
The above function $v$ is well-defined.
Furthermore, $v$ is a valuation which is trivial on $\C$.
\end{cor}

\begin{proof}
For the well-definedness, we weed to check 
\begin{enumerate}
\item for $x,y \in R_d, y \neq 0$,  $-\iota(x) + \iota(y) \in  \Z \cup \{\infty\}$.
\item if $x/y =x'/y' \in K(X)$, it holds that  $-\iota(x) + \iota(y) = -\iota(x') + \iota(y') $.
\end{enumerate}

As in \autoref{def_function}, $\iota(y) \in \N$ if $y \neq 0$.
Since $ -\iota(x) $ is in $  \Z \cup \{\infty\} $, 
we have $-\iota(x) + \iota(y) \in  \Z \cup \{\infty\}$.
Thus (1) follows.

\vspace{2mm}
For (2), if $x/y =x'/y'  \in K(X)$,
we have $ xy' = x'y \in R$.
Then 
\[
\iota(x)+\iota(y') = \iota(xy') = \iota(x'y) =\iota(x') +\iota(y) 
\]
by \autoref{prop_app}.
Hence $-\iota(x) + \iota(y) = -\iota(x') + \iota(y') $ holds.
Thus $v $ is well-defined.

\vspace{2mm}

By \autoref{lem_multiplicative}, \autoref{lem_socle_derivation} and \autoref{prop_app},
we can check that $v$ satisfies the definition of valuation,
i.e.\
\begin{itemize}
\item $v(0) = \infty $ and $v(x) \neq \infty$ for $x \in K(X) \setminus 0$.
\item $v(x+y) \geq \min \{ v(x),v(y)\}$ for $x,y \in K(X)$, with equality if $ v(x) \neq v(y)$.
\item $v(xy) =v(x)+v(y)$  for $x,y \in K(X)$.
\item $v(a) =0$ for $a \in \C \setminus 0$.
\end{itemize}
\end{proof}

\begin{ex}\label{ex_toric_valuation}
Let $P$ be a reflexive polytope with a unique unipotent root $m=(0,-1) \in M' \times \Z$ such that
\[
P=\{ (u',t) \in F' \times \R \, | \,  -1 \leq t \leq h(u') \}
\]
for some $F',h$ as in \autoref{sec_ex}.
In this case,
the valuation induced by the Socle filtration $\sf$ is the toric valuation corresponding to $(0,-1) \in N' \times \Z$.
We note that this is \emph{not} the divisorial valuation $\ord_{D}$, which corresponds to $(0,1) \in N' \times \Z$, for the prime divisor $D \subset X$ corresponding to the facet $F=F' \times \{-1\}$ of $P$.

For example, for the singular del Pezzo surface in \autoref{subsec_surface},
the valuation $v$ is nothing but the divisorial valuation $\ord_E$,
where $E$ is the exceptional divisor of the blow-up of the ordinary double point in $X$. 

Recall that the function $g$ in \autoref{sec_ex} corresponding to the Socle filtration is not only concave but also affine,
contrary to the convex function $f$.
The affineness is due to \autoref{cor_valuation}.
\end{ex}

The author does not know whether or not the valuation $v$ is the divisorial valuation for some prime divisor over $X$ in general.

\end{document}